\documentclass{amsart}

\usepackage[T1]{fontenc}
\usepackage{lmodern}
\usepackage{amssymb,mathtools}
\usepackage{microtype}
\usepackage{booktabs}
\usepackage{enumitem}
\usepackage{url}
\usepackage[hidelinks]{hyperref}

\hypersetup{
  pdftitle={Optimal differentiability of isotropic positive definite functions on even-dimensional spheres},
  pdfauthor={Yan Ge},
  pdfsubject={Positive definite functions on spheres and optimal differentiability},
  pdfkeywords={positive definite functions, spheres, Gegenbauer polynomials, differentiability, dimension walks}
}

\numberwithin{equation}{section}

\newtheorem{theorem}{Theorem}[section]
\newtheorem{proposition}[theorem]{Proposition}
\newtheorem{lemma}[theorem]{Lemma}
\newtheorem{corollary}[theorem]{Corollary}
\theoremstyle{remark}

\newcommand{\Nzero}{\mathbb N_0}
\newcommand{\R}{\mathbb R}
\newcommand{\Sph}{\mathbb S}
\newcommand{\dd}{\,\mathrm d}
\newcommand{\abs}[1]{\left\lvert #1\right\rvert}
\newcommand{\norm}[1]{\left\lVert #1\right\rVert}
\newcommand{\ip}[2]{\left\langle #1,#2\right\rangle}
\newcommand{\IS}{\mathcal I_{\!S}}
\newcommand{\Psit}{\widetilde{\Psi}}
\newcommand{\Sd}{\mathcal S}

\begin{document}

\title[Optimal differentiability on even-dimensional spheres]{Optimal differentiability of isotropic positive definite functions\\
on even-dimensional spheres}
\author{Yan Ge}
\address{Department of Mathematics, School of Science, North China University of Technology, Beijing 100144, China}
\email{yge@ncut.edu.cn}
\date{}

\begin{abstract}
We prove optimality of the differentiability bound for isotropic positive definite functions on every even-dimensional sphere. If the even continuation of such a function on the $d$-dimensional sphere is $2k$ times differentiable at zero, then the function has $2k+\lfloor(d-1)/2\rfloor$ continuous interior derivatives; previously, optimality was known only in odd dimensions. We construct a function on the two-dimensional sphere whose first derivative does not exist at the equator and transfer it to all even dimensions by turning bands and spherical montée. The resulting examples are strictly positive definite, have $2k$ but not $2k+2$ derivatives at zero, and are not positive definite in the next dimension.
\end{abstract}
\maketitle

\section{Introduction}

Isotropic positive definite functions arise as covariance models for spherical random fields and as zonal kernels in approximation theory; see the survey~\cite{Gneiting2013}.

The regularity problem has a Euclidean analogue. A radial positive definite function on $\R^d$ that is $2k$ times differentiable at zero has $2k+\lfloor(d-1)/2\rfloor$ continuous derivatives away from the origin, and this order is optimal~\cite{Gneiting1999}. On spheres, the analogous arguments use the spherical mont\'ee, descente, and turning-bands operators, paralleling dimension walks introduced by Matheron~\cite{Matheron1965,Matheron1972}; see also~\cite{EmeryPeronPorcu2022}.

For an integer $d\ge1$, define
$\Sph^d:=\{x\in\R^{d+1}:\norm{x}=1\}$ and
$\vartheta(x,y):=\arccos\langle x,y\rangle$.
A continuous function $\psi:[0,\pi]\to\R$ belongs to $\Psi_d$ if $\psi(0)=1$ and
\begin{equation}\label{eq:positive-definite-definition}
  \sum_{i=1}^N\sum_{j=1}^N c_i c_j
  \psi\bigl(\vartheta(x_i,x_j)\bigr)\ge0
\end{equation}
for every integer $N\ge1$, every $x_1,\ldots,x_N\in\Sph^d$, and every $c_1,\ldots,c_N\in\R$. The associated isotropic kernel is strictly positive definite if the inequality in \eqref{eq:positive-definite-definition} is strict whenever the points $x_1,\ldots,x_N$ are pairwise distinct and $(c_1,\ldots,c_N)\ne0$. Define
\begin{equation*}
  \Psi_d^+:=\{\psi\in\Psi_d:\psi(\vartheta(x,y))
  \text{ is strictly positive definite on }\Sph^d\}.
\end{equation*}
For an interval $I$ and an integer $r\ge0$, let $C^r(I)$ denote the
space of functions $f:I\to\mathbb R$ for which
$f,f',\ldots,f^{(r)}$ are continuous on $I$.
Here $C^0(I)$ means continuity. The existence of $f^{(r)}(\theta_0)$ means that $f^{(r-1)}$ is defined near $\theta_0$ and has a finite two-sided derivative there. At zero, differentiability of $\psi$ refers to the even continuation
$\psi_e(t):=\psi(|t|)$, $-\pi<t<\pi$.
This convention agrees with Tr\"ubner and Ziegel~\cite[Sec.~1]{TrubnerZiegel2017}.

Ziegel~\cite[Theorem~1.2]{Ziegel2014} proved that every function in $\Psi_d$ has $\lfloor(d-1)/2\rfloor$ continuous derivatives on $(0,\pi)$. Furthermore, Tr\"ubner and Ziegel~\cite[Theorem~1.1]{TrubnerZiegel2017} obtained the following refinement.

\begin{theorem}\label{thm:main-smoothness}
Let $d\ge1$ and $k\ge0$. If $\psi\in\Psi_d$ and $\psi_e$ is $2k$ times differentiable at $0$, then
\begin{equation*}
  \psi\in C^{\,2k+\lfloor(d-1)/2\rfloor}(0,\pi).
\end{equation*}
\end{theorem}

The order in Theorem~\ref{thm:main-smoothness} is optimal in every odd
dimension; see Ziegel~\cite[Sec.~4]{Ziegel2014} and Tr\"ubner--Ziegel
\cite[Sec.~4]{TrubnerZiegel2017}. Those examples start with a
nondifferentiable function in $\Psi_1$ and use dimension walks that
raise the dimension by two. Hence the even-dimensional problem reduces
to finding a function in $\Psi_2$ that is not continuously
differentiable on $(0,\pi)$; see Ziegel~\cite[Sec.~4.2]{Ziegel2014} and
Tr\"ubner--Ziegel~\cite[Sec.~1]{TrubnerZiegel2017}. The construction of
such a function has remained an open problem.
The main difficulty is pointwise: one must prove that the two-dimensional
example itself fails to be differentiable at an interior point, rather than
merely observe divergence of a formally differentiated series.

We construct a function in $\Psi_2$ whose first derivative does not exist at
the equator. The resulting statement is as follows.

\begin{theorem}[Optimality in even dimensions]\label{thm:main-even}
Let $d=2m\ge2$ and $k\ge0$. There exists a function $\psi_{d,k}\in\Psi_d^+$ such that
\begin{enumerate}[label=\textup{(\roman*)}]
  \item the even continuation of $\psi_{d,k}$ is $2k$ times differentiable at $0$;
  \item $\psi_{d,k}\in C^{2k+m-1}(0,\pi)$;
  \item $\psi_{d,k}^{(2k+m)}(\pi/2)$ does not exist;
  \item the even continuation of $\psi_{d,k}$ is not $2k+2$ times differentiable at $0$;
  \item $\psi_{d,k}\notin\Psi_{d+1}$.
\end{enumerate}
Thus $2k+m-1=2k+\lfloor(d-1)/2\rfloor$ is the largest integer differentiability order guaranteed by the hypothesis at zero.
\end{theorem}

Section~2 recalls Schoenberg expansions and the moment criterion at the origin;
Section~3 records the dimension-walk formulas. Sections~4--6 construct the
two-dimensional seed, transfer it to all even dimensions, and identify the
optimal differentiability index.

\section{Schoenberg expansions and differentiability at the origin}

The classes $\Psi_d$ are characterized by Gegenbauer expansions. We recall this characterization and the moment condition on the coefficients corresponding to differentiability at zero. Throughout, define
$\Nzero:=\{0,1,2,\ldots\}$.

\subsection{The classes \texorpdfstring{$\Psi_d$}{Psi d} and their Schoenberg sequences}

The classes $\Psi_d$ and $\Psi_d^+$ were defined in the Introduction. We
also use the nonnormalized cone
$\Psit_d:=\{a\psi:a\ge0,\ \psi\in\Psi_d\}$.
Thus $f\in\Psit_d$ is a continuous isotropic positive semidefinite function on $\Sph^d$ without the normalization $f(0)=1$. The cone $\Psit_d$ is used in the dimension-walk formulas, where normalization obscures the coefficient shifts.

For $\lambda>0$, let $C_n^\lambda$ be the Gegenbauer polynomial and set
\begin{equation*}
  R_n^\lambda(x):=\frac{C_n^\lambda(x)}{C_n^\lambda(1)},
  \qquad -1\le x\le1.
\end{equation*}
For $d\ge1$, define $\lambda_d:=(d-1)/2$. For $d=1$ we use the
limiting convention $R_n^0(\cos\theta):=\cos(n\theta)$.

The following coefficient characterization is Schoenberg's theorem~\cite{Schoenberg1942}.

\begin{theorem}
A continuous function $\psi:[0,\pi]\to\R$ belongs to $\Psi_d$ if and only if
\begin{equation}\label{eq:Schoenberg-series}
  \psi(\theta)=\sum_{n=0}^\infty b_{n,d}
  R_n^{\lambda_d}(\cos\theta),
  \qquad
  b_{n,d}\ge0,
  \qquad
  \sum_{n=0}^\infty b_{n,d}=1,
\end{equation}
with the cosine convention when $d=1$.
\end{theorem}

The sequence $(b_{n,d})_{n\ge0}$ is called the $d$-Schoenberg sequence of $\psi$. The standard bound for normalized Gegenbauer polynomials~\cite[Eq.~18.14.4]{DLMF}
is $\abs{R_n^\lambda(x)}\le1$ for $\lambda>0$ and $-1\le x\le1$.
For $d=1$, the corresponding estimate is $|\cos(n\theta)|\le1$.
Hence the Weierstrass test gives uniform convergence of
\eqref{eq:Schoenberg-series} on $[0,\pi]$ in every dimension. For
$d\ge2$, the criterion of Chen, Menegatto and
Sun~\cite{ChenMenegattoSun2003} says that $\psi\in\Psi_d^+$ if and only
if positive Schoenberg coefficients occur for infinitely many even and
infinitely many odd degrees.

When a nonnegative summable sequence $\beta=(\beta_n)_{n\ge0}$ is given, we write
\begin{equation*}
  \Sd_d(\beta;\theta)
  :=\sum_{n=0}^\infty \beta_n R_n^{\lambda_d}(\cos\theta),
\end{equation*}
again with the cosine convention for $d=1$. Then
$\Sd_d(\beta;\cdot)\in\Psit_d$ and
$\Sd_d(\beta;0)=\sum_{n=0}^\infty\beta_n$.

\subsection{A moment criterion at the origin}

The following criterion is due to Tr\"ubner and Ziegel~\cite[Lemma~2.1(a)]{TrubnerZiegel2017}. The corresponding Euclidean result is Gneiting~\cite[Lemma~3]{Gneiting1999}; J\"ager~\cite{Jager2019} treats the related Hilbert-sphere moment problem.

\begin{theorem}\label{thm:moment}
Let $d\ge1$, $q\ge1$, and let $\psi\in\Psi_d$ have Schoenberg sequence $(b_{n,d})_{n\ge0}$. The even continuation $\psi_e(t):=\psi(|t|)$ has a finite $2q$-th derivative at $0$---that is, $\psi_e^{(2q-1)}$ is defined near $0$ and has a finite two-sided derivative there---if and only if
\begin{equation}
  \sum_{n=0}^\infty n^{2q}b_{n,d}<\infty.
\end{equation}
\end{theorem}

For $q=0$, the corresponding condition is only $\sum_{n\ge0}b_{n,d}=1$, and ``zero times differentiable'' means continuous.

\section{Dimension walks and transfer of regularity}

Turning bands relates dimensions $d$ and $d+2$, while the spherical mont\'ee relates dimensions $D$ and $D-2$. We record the identities, coefficient formulas, and moment estimates needed below, and prove the two pointwise implications used in Section~\ref{sec:even}.

\subsection{Turning bands}

For a sequence $\beta=(\beta_n)_{n\ge0}$ define the left shift by
$(T\beta)_n:=\beta_{n+1}$, $n\ge0$.
This is the sequence operation denoted by $\beta\circ\tau_{-1}$ in Ziegel~\cite{Ziegel2014} and Tr\"ubner--Ziegel~\cite{TrubnerZiegel2017}.
It preserves nonnegativity. Since
$\sum_{n=0}^\infty(T\beta)_n=\sum_{n=1}^\infty\beta_n
\le\sum_{n=0}^\infty\beta_n$, it also preserves summability.

\begin{proposition}[Turning-bands identity]\label{prop:turning-bands}
Let $d\ge1$ and let $\beta$ be nonnegative and summable. Put
$F(\theta):=\Sd_d(\beta;\theta)$ and
$G(\theta):=\Sd_{d+2}(T\beta;\theta)$.
Then $G\in C^1(0,\pi)$ and, for $0<\theta<\pi$,
\begin{equation}
  F(\theta)
  =\beta_0+\cos\theta\,G(\theta)
    +\frac{\sin\theta}{d}G'(\theta).\label{eq:turning-diff}
\end{equation}
\end{proposition}

\begin{proof}
If $G\equiv0$, then $\beta_n=0$ for every $n\ge1$, so $F\equiv\beta_0$
and the assertion is immediate. Otherwise
$G(0)=\sum_{n=1}^{\infty}\beta_n>0$ and $G/G(0)\in\Psi_{d+2}$.
Ziegel's interior regularity theorem~\cite[Theorem~1.2]{Ziegel2014} gives
$G\in C^{\lfloor(d+1)/2\rfloor}(0,\pi)$, and hence $G\in C^1(0,\pi)$.
With $R_n^{\lambda_d}=C_n^{\lambda_d}/C_n^{\lambda_d}(1)$ and
$(T\beta)_n=\beta_{n+1}$, identity~\eqref{eq:turning-diff} is exactly
Ziegel~\cite[Proposition~4.4]{Ziegel2014} in the present normalization.
\end{proof}

\begin{lemma}\label{lem:turning-singularity}
Let $\theta_0\in(0,\pi)$ and $r\ge2$. In the notation of Proposition~\ref{prop:turning-bands}, assume that $G\in C^{r-1}$ in a neighborhood of $\theta_0$. If $F^{(r-1)}(\theta_0)$ does not exist, then $G^{(r)}(\theta_0)$ does not exist.
\end{lemma}

\begin{proof}
Assume, to the contrary, that $G^{(r)}(\theta_0)$ exists. Since $G\in C^{r-1}$ near $\theta_0$, equation~\eqref{eq:turning-diff} gives $F\in C^{r-2}$ there. The pointwise existence of $G^{(r)}(\theta_0)$ then permits one further differentiation at $\theta_0$. Since $r-1\ge1$, the constant term $\beta_0$ disappears, and Leibniz' rule gives
\begin{align*}
  F^{(r-1)}(\theta_0)
  &=\sum_{\ell=0}^{r-1}\binom{r-1}{\ell}
    (\cos\theta)^{(r-1-\ell)}\big|_{\theta=\theta_0}
    G^{(\ell)}(\theta_0)\\
  &\quad+\frac1d\sum_{\ell=0}^{r-1}\binom{r-1}{\ell}
    (\sin\theta)^{(r-1-\ell)}\big|_{\theta=\theta_0}
    G^{(\ell+1)}(\theta_0).
\end{align*}
Every derivative on the right exists under the contradictory
assumption. Hence $F^{(r-1)}(\theta_0)$ exists, contrary to the
hypothesis.
\end{proof}

\subsection{The spherical mont\'ee}

Put $A(f):=\int_0^\pi \sin t\,f(t)\dd t$, where $f$ is continuous.
Assume that $A(f)\ne0$ and define
\begin{equation*}
  (\IS f)(\theta)
  :=\frac{\displaystyle\int_\theta^\pi\sin t\,f(t)\dd t}
  {A(f)},\qquad 0\le\theta\le\pi.
\end{equation*}
The normalization $(\IS f)(0)=1$ follows directly from the definition.
Moreover, the fundamental theorem of calculus gives, for
$0<\theta<\pi$,
\begin{equation}\label{eq:montee-derivative}
  (\IS f)'(\theta)=-\frac{\sin\theta}{A(f)}f(\theta).
\end{equation}
The following proposition records the mont\'ee formula and the strict positive definiteness consequence needed below.

\begin{proposition}\label{prop:montee-positive}
Let $D\ge4$. Suppose $f\in\Psi_D$, $f\ge0$, and let $(c_{n,D})_{n\ge0}$ be its $D$-Schoenberg sequence. Then $A(f)>0$, $\IS f\in\Psi_{D-2}$, and for $n\ge1$ its $(D-2)$-Schoenberg coefficients are
\begin{equation}\label{eq:montee-coeff}
  a_{n,D-2}
  =\frac{(D-2)c_{n-1,D}}
  {n(n+D-3)A(f)}.
\end{equation}
If $c_{n,D}>0$ for all sufficiently large $n$, then the Schoenberg coefficients of $\IS f$ are positive for all sufficiently large positive degrees; in particular $\IS f\in\Psi_{D-2}^+$.
\end{proposition}

\begin{proof}
Since $f(0)=1$ and $f$ is continuous, there is a number
$\delta\in(0,\pi)$ such that $f(t)\ge1/2$ for $0\le t\le\delta$.
The nonnegativity of $f$ therefore gives
$A(f)\ge\frac12\int_0^\delta\sin t\,\dd t>0$.
In the notation of Tr\"ubner--Ziegel, their normalizing constant
$G_1(\Nzero)$ is $A(f)$. Thus the membership $\IS f\in\Psi_{D-2}$ and
formula~\eqref{eq:montee-coeff} follow from
Tr\"ubner--Ziegel~\cite[Proposition~3.3(a),(c)]{TrubnerZiegel2017};
see also Beatson--zu Castell~\cite[Theorem~2.2(a)(i),(iii)]{BeatsonZuCastell2017}.
If $c_{n,D}>0$ for every sufficiently large $n$, formula
\eqref{eq:montee-coeff} gives $a_{n,D-2}>0$ for every sufficiently
large positive degree. Since $D-2\ge2$, the criterion of Chen,
Menegatto and Sun~\cite{ChenMenegattoSun2003} then yields
$\IS f\in\Psi_{D-2}^+$.
\end{proof}

The next lemma is the moment form of Tr\"ubner--Ziegel~\cite[Lemma~3.10]{TrubnerZiegel2017}. We include the coefficient argument used later.

\begin{lemma}\label{lem:montee-pole}
Under the assumptions of Proposition~\ref{prop:montee-positive}, let $q\in\Nzero$ and suppose the even continuation of $f$ is $2q$ times differentiable at $0$. Then the even continuation of $\IS f$ is $2q+2$ times differentiable at $0$.
\end{lemma}

\begin{proof}
Put $S_q:=\sum_{n=0}^\infty(n+1)^{2q}c_{n,D}$.
If $q=0$, normalization gives $S_0=1$. If $q\ge1$, then
Theorem~\ref{thm:moment} and
$(n+1)^{2q}\le2^{2q-1}(n^{2q}+1)$ give
\begin{align*}
  S_q
  &\le 2^{2q-1}\sum_{n=0}^\infty(n^{2q}+1)c_{n,D}=2^{2q-1}\left(
    \sum_{n=0}^\infty n^{2q}c_{n,D}+1\right)<\infty.
\end{align*}
Using \eqref{eq:montee-coeff} and $n+D-3\ge n$ gives, for $n\ge1$,
$n^{2q+2}a_{n,D-2}=\frac{D-2}{A(f)}
\frac{n^{2q+1}}{n+D-3}c_{n-1,D}
\le \frac{D-2}{A(f)}n^{2q}c_{n-1,D}$.
Consequently, after setting $m:=n-1$,
\begin{equation*}
  \sum_{n=1}^{\infty}n^{2q+2}a_{n,D-2}
  \le \frac{D-2}{A(f)}
      \sum_{m=0}^{\infty}(m+1)^{2q}c_{m,D}
  =\frac{D-2}{A(f)}S_q<\infty.
\end{equation*}
Theorem~\ref{thm:moment} in dimension $D-2$ now gives the conclusion.
\end{proof}

The following implication is the pointwise version of the qualitative relation used in Tr\"ubner--Ziegel~\cite[Eq.~(13)]{TrubnerZiegel2017}.

\begin{lemma}\label{lem:montee-singularity}
Let $f$ be continuous, $A(f)\ne0$, $\theta_0\in(0,\pi)$, and $r\ge1$. If $f^{(r)}(\theta_0)$ does not exist, then $(\IS f)^{(r+1)}(\theta_0)$ does not exist.
\end{lemma}

\begin{proof}
Write $h:=\IS f$ and suppose that $h^{(r+1)}(\theta_0)$ exists. By the
pointwise convention in the Introduction, $h^{(r)}$ is then defined near
$\theta_0$. Equation~\eqref{eq:montee-derivative} gives there
$f(\theta)=-A(f)(\sin\theta)^{-1}h'(\theta)$. Repeated differentiation
therefore gives, in that neighborhood,
\begin{equation*}
  f^{(r-1)}(\theta)
  =-A(f)\sum_{\ell=0}^{r-1}\binom{r-1}{\ell}
  \left((\sin\theta)^{-1}\right)^{(r-1-\ell)}
  h^{(\ell+1)}(\theta).
\end{equation*}
Since $(\sin\theta)^{-1}$ is smooth near $\theta_0$, differentiating
this identity once more at $\theta_0$ gives
\begin{equation*}
  f^{(r)}(\theta_0)
  =-A(f)\sum_{\ell=0}^{r}\binom r\ell
  \left((\sin\theta)^{-1}\right)^{(r-\ell)}\big|_{\theta=\theta_0}
  h^{(\ell+1)}(\theta_0).
\end{equation*}
Every term on the right exists, contradicting the hypothesis that $f^{(r)}(\theta_0)$ does not exist.
\end{proof}

\section{The two-dimensional example}

Odd-dimensional examples start from a nondifferentiable positive definite
function on the circle~\cite{Ziegel2014,TrubnerZiegel2017}. Since the
dimension walks change dimension by two, we instead construct a seed on
$\Sph^2$. Its sparse Schoenberg sequence is supported on odd degrees for
which the equatorial derivatives have one sign, with coefficients at the
threshold between convergence of the series and divergence of its formal
derivative.

\subsection{The sparse Schoenberg sequence}

Let $P_n:=R_n^{1/2}$ denote the Legendre polynomial normalized by
$P_n(1)=1$; this is the specialization $C_n^{1/2}=P_n$
\cite[Eq.~18.7.9]{DLMF}. Here $P_n'$ denotes differentiation with
respect to the polynomial variable. At the equator,
$\frac{\dd}{\dd\theta}P_n(\cos\theta)|_{\theta=\pi/2}=-P_n'(0)$.
For $n=4j+1$ the numbers $P_n'(0)$ have one sign and order $j^{1/2}$.
Coefficients of order $j^{-\alpha}$ are summable for $\alpha>1$, while
the derivative series diverges for $\alpha\le3/2$; we take the endpoint
$\alpha=3/2$. Lemma~\ref{lem:Legendre-derivative} gives the precise estimates.

For a real sequence $(u_j)$ and a positive sequence $(v_j)$, the notation
$u_j=O(v_j)$ means that $|u_j|/v_j$ is bounded for all sufficiently large
$j$. For $s>1$, define the Riemann zeta function by
$\zeta(s):=\sum_{\ell=1}^\infty \ell^{-s}$.
For $j\in\Nzero$, define
\begin{equation}\label{eq:aj}
  a_j
  :=\frac{(j+1)^{-3/2}}
  {\displaystyle\sum_{\ell=0}^\infty(\ell+1)^{-3/2}}
  =\frac{(j+1)^{-3/2}}{\zeta(3/2)}.
\end{equation}
Since the denominator in \eqref{eq:aj} equals $\zeta(3/2)$, one has $\sum_{j=0}^\infty a_j=1$. Define
\begin{equation*}
  b_n:=
  \begin{cases}
    a_{(n-1)/4},&n\equiv1\pmod4,\\
    0,&\text{otherwise}.
  \end{cases}
\end{equation*}
Then $b_n\ge0$ and $\sum_n b_n=1$. Schoenberg's theorem gives
\begin{equation}\label{eq:phi-seed}
  \phi(\theta):=\sum_{j=0}^\infty a_jP_{4j+1}(\cos\theta)\in\Psi_2.
\end{equation}
Uniform convergence follows from $\abs{P_n(x)}\le1$ on $[-1,1]$~\cite[Eq.~18.14.4]{DLMF}. Formula~\cite[Eq.~18.5.10]{DLMF}, with $\lambda=1/2$, shows that $P_{4j+1}(0)=0$; hence $\phi(\pi/2)=0$.
The same finite expansion gives $P_{4j+1}(-x)=-P_{4j+1}(x)$, and
therefore $\phi(\pi/2+h)=-\phi(\pi/2-h)$. Thus $\phi$ is
antisymmetric about the equator.

\begin{lemma}\label{lem:Legendre-derivative}
For every $j\ge0$,
\begin{equation}\label{eq:Legendre-exact}
  P_{4j+1}'(0)
  =\frac{(4j+1)!}{2^{4j}((2j)!)^2}
  =(4j+1)\frac{\binom{4j}{2j}}{2^{4j}}>0.
\end{equation}
Moreover,
\begin{equation}\label{eq:Legendre-asympt}
  P_{4j+1}'(0)
  =\sqrt{\frac8\pi}\,j^{1/2}\bigl(1+O(j^{-1})\bigr),
  \qquad j\to\infty\text{ through }j\ge1,
\end{equation}
and therefore
\begin{equation}\label{eq:divergent-derivative-sum}
  \sum_{j=0}^\infty a_jP_{4j+1}'(0)=+\infty.
\end{equation}
\end{lemma}

\begin{proof}
The finite power series for the Gegenbauer polynomials~\cite[Eq.~18.5.10]{DLMF}, specialized to $\lambda=1/2$, gives
\begin{equation*}
  P_{2\ell+1}'(0)=(-1)^\ell
  \frac{(2\ell+1)!}{2^{2\ell}(\ell!)^2}.
\end{equation*}
Taking $\ell=2j$ proves \eqref{eq:Legendre-exact}. Stirling's formula~\cite[Eq.~5.11.3]{DLMF} gives
\begin{equation*}
  \binom{4j}{2j}
  =\frac{2^{4j}}{\sqrt{2\pi j}}
  \bigl(1+O(j^{-1})\bigr)
\end{equation*}
and hence \eqref{eq:Legendre-asympt}. Combining this asymptotic with \eqref{eq:aj} gives
\begin{equation*}
  a_jP_{4j+1}'(0)
  =\frac{1}{\zeta(3/2)}\sqrt{\frac8\pi}\,
  \frac{1+O(j^{-1})}{j},
  \qquad j\to\infty,
\end{equation*}
so \eqref{eq:divergent-derivative-sum} follows from divergence of the harmonic series.
\end{proof}

\subsection{Poisson regularization}

Fix $e\in\Sph^2$ and define the zonal function
$F(\xi):=\phi(\arccos\ip{e}{\xi})$, $\xi\in\Sph^2$.
Let $\dd\sigma$ denote standard surface-area measure on $\Sph^2$, so that
$\sigma(\Sph^2)=4\pi$. For $0<r<1$, let
\begin{equation}\label{eq:Poisson-extension}
  U(r,\xi):=\int_{\Sph^2}P_r(\xi,\eta)F(\eta)\dd\sigma(\eta),
  \qquad
  P_r(\xi,\eta):=\frac{1-r^2}{4\pi\abs{r\xi-\eta}^3}.
\end{equation}
For fixed $e$, the function $\xi\mapsto P_n(\langle e,\xi\rangle)$ is a zonal spherical harmonic of degree $n$. The Poisson formula in \eqref{eq:Poisson-extension} and its action as the multiplier $r^n$ on spherical harmonics of degree $n$ follow from Dai and Xu~\cite[Definition~2.2.3 and Lemma~2.2.4]{DaiXu2013}, after accounting for our use of unnormalized surface-area measure. Since $P_r(\xi,\cdot)\in L^1(\dd\sigma)$ for each fixed $r<1$, uniform convergence of \eqref{eq:phi-seed} permits termwise Poisson integration. Hence the one-variable part of $U$ is
\begin{equation*}
  \phi_r(\theta)
  :=\sum_{j=0}^\infty a_jr^{4j+1}P_{4j+1}(\cos\theta).
\end{equation*}
Thus
\begin{equation}\label{eq:U-zonal}
  U(r,\xi)=\phi_r\bigl(\arccos\ip{e}{\xi}\bigr).
\end{equation}
Since $\theta\mapsto P_{4j+1}(\cos\theta)$ is a trigonometric polynomial of degree $4j+1$ and has supremum norm at most one, Bernstein's inequality~\cite[Theorem~4.1.3]{DaiXu2013} gives
\begin{equation*}
  \sup_{0\le\theta\le\pi}
  \left|\frac{\dd}{\dd\theta}P_{4j+1}(\cos\theta)\right|
  \le4j+1.
\end{equation*}
Since $\sum_j a_j(4j+1)r^{4j+1}<\infty$, the differentiated series
converges uniformly. In particular, for $0<r<1$,
\begin{equation*}
  \phi_r'\left(\frac\pi2\right)
  =-\sum_{j=0}^{\infty}a_jr^{4j+1}P_{4j+1}'(0).
\end{equation*}
Every term in the last sum is nonnegative and increases to
its value at $r=1$. Thus monotone convergence and
Lemma~\ref{lem:Legendre-derivative} give
\begin{align*}
  \lim_{r\to1^-}\left(-\phi_r'\left(\frac\pi2\right)\right)
  &=\lim_{r\to1^-}\sum_{j=0}^\infty
    a_jr^{4j+1}P_{4j+1}'(0)\\
  &=\sum_{j=0}^\infty a_jP_{4j+1}'(0)=+\infty.
\end{align*}
Equivalently,
\begin{equation}\label{eq:phi-r-minus-infty}
  \lim_{r\to1^-}\phi_r'\left(\frac\pi2\right)=-\infty.
\end{equation}

To pass from \eqref{eq:phi-r-minus-infty} to the original boundary function, we use the following local statement for spherical Poisson integrals.

For $\xi_0\in\Sph^2$, define the tangent space and geodesic distance by
\begin{equation*}
  T_{\xi_0}\Sph^2:=\{v\in\R^3:\ip{v}{\xi_0}=0\},
  \qquad
  \operatorname{dist}(\eta,\xi_0):=\arccos\ip{\eta}{\xi_0}.
\end{equation*}
If $v\in T_{\xi_0}\Sph^2\setminus\{0\}$, define the great-circle curve
\begin{equation*}
  \gamma_v(s):=\cos(s\norm v)\xi_0
  +\sin(s\norm v)\frac{v}{\norm v}
\end{equation*}
and, whenever it exists, the tangential derivative
\begin{equation*}
  D_vG(\xi_0):=\frac{\dd}{\dd s}G(\gamma_v(s))\bigg|_{s=0}.
\end{equation*}
Set $D_0G(\xi_0):=0$.

For $\eta$ near $\xi_0$, put $t:=\operatorname{dist}(\eta,\xi_0)$.
When $t>0$, there is a unique unit vector $u\in T_{\xi_0}\Sph^2$ such
that $\eta=\cos t\,\xi_0+\sin t\,u$.
We say that $G$ is differentiable at $\xi_0$ if there is a vector $b\in T_{\xi_0}\Sph^2$ for which
\begin{equation*}
  G(\eta)=G(\xi_0)+t\ip{b}{u}+o(t)
  \qquad\text{as }\eta\to\xi_0,
\end{equation*}
where the little-$o$ is uniform over all approach directions $u$.
In that case $D_vG(\xi_0)=\ip{b}{v}$ for every $v\in T_{\xi_0}\Sph^2$.

\begin{lemma}\label{lem:Poisson-boundary}
Let $G\in C(\Sph^2)$ and
\begin{equation*}
  V(r,\xi):=\int_{\Sph^2}P_r(\xi,\eta)G(\eta)\dd\sigma(\eta).
\end{equation*}
Fix $\xi_0\in\Sph^2$. If $G$ is differentiable at $\xi_0$ as a function on the manifold, then for every $v\in T_{\xi_0}\Sph^2$,
\begin{equation}\label{eq:Poisson-boundary-convergence}
  \lim_{r\to1^-}D_vV(r,\xi_0)=D_vG(\xi_0).
\end{equation}
\end{lemma}

\begin{proof}
For fixed $r<1$, the inequality $|r\xi-\eta|\ge1-r$ holds for
$\xi,\eta\in\Sph^2$. If $w\in T_\xi\Sph^2$, direct differentiation
therefore gives the uniform bounds
\begin{align*}
  |P_r(\xi,\eta)|
  &\le\frac{1-r^2}{4\pi(1-r)^3},&
  |D_wP_r(\xi,\eta)|
  &\le\frac{3r(1-r^2)\norm w}{4\pi(1-r)^5}.
\end{align*}
The functions multiplied by these kernels below are bounded on the
compact sphere. Thus differentiation under the integral is justified
for every fixed $r<1$.

\smallskip
\noindent\emph{Step 1: Reduction to a remainder vanishing to first order.}
Let $b\in T_{\xi_0}\Sph^2$ represent the differential of $G$ at
$\xi_0$, and put $\ell(\eta):=G(\xi_0)+\ip{b}{\eta}$ and
$H(\eta):=G(\eta)-\ell(\eta)$. If
$\eta=\cos t\,\xi_0+\sin t\,u$, then $\ip{b}{\xi_0}=0$ and
$\ell(\eta)=G(\xi_0)+\sin t\,\ip{b}{u}$. Subtracting this identity
from the differentiability expansion of $G$ gives
\begin{align*}
  H(\eta)
  &=\bigl(G(\eta)-G(\xi_0)-t\ip{b}{u}\bigr)
    +(t-\sin t)\ip{b}{u}\\
  &=o(t)+O(t^3)=o(t).
\end{align*}
In particular,
\begin{equation}\label{eq:H-little-o}
  H(\xi_0)=0,
  \qquad
  \frac{|H(\eta)|}{\operatorname{dist}(\eta,\xi_0)}\longrightarrow0
  \quad\text{as }\eta\to\xi_0.
\end{equation}
Constants and degree-one spherical harmonics have Poisson multipliers
$1$ and $r$, respectively~\cite[Sec.~2.2]{DaiXu2013}. The Poisson
extension of $\ell$ is therefore $G(\xi_0)+r\ip{b}{\xi}$. Define
\begin{equation*}
  V_H(r,\xi):=\int_{\Sph^2}P_r(\xi,\eta)H(\eta)\dd\sigma(\eta).
\end{equation*}
Then
\begin{equation*}
  D_vV(r,\xi_0)=r\ip{b}{v}+D_vV_H(r,\xi_0).
\end{equation*}
The first term tends to $\ip{b}{v}=D_vG(\xi_0)$. It remains to prove
that $D_vV_H(r,\xi_0)$ tends to zero.

\smallskip
\noindent\emph{Step 2: Estimate for the differentiated kernel.}
First suppose that $\norm v=1$. Since $\ip{\xi_0}{v}=0$, direct differentiation of \eqref{eq:Poisson-extension} gives
\begin{equation}\label{eq:Poisson-tangent-derivative}
  D_vP_r(\xi_0,\eta)
  =\frac{3r(1-r^2)}{4\pi}
  \frac{\ip{\eta}{v}}{\abs{r\xi_0-\eta}^5}.
\end{equation}
Set $\varepsilon:=1-r$ and
$t:=\operatorname{dist}(\eta,\xi_0)$. For $r\ge1/2$ and
$0\le t\le1$, one has
\begin{align*}
  |\ip{\eta}{v}|&\le\sin t\le t,\\
  |r\xi_0-\eta|^2
  &=(1-r)^2+2r(1-\cos t)
  \ge\varepsilon^2+\frac{2}{\pi^2}t^2.
\end{align*}
Here we used $1-\cos t\ge2t^2/\pi^2$ for $0\le t\le\pi$. In the estimates below, $C>0$ denotes an absolute constant whose value may change from line to line. Since $1-r^2\le2\varepsilon$, it follows that
\begin{equation*}
  \abs{D_vP_r(\xi_0,\eta)}
  \le C\frac{\varepsilon t}{(\varepsilon^2+t^2)^{5/2}}
\end{equation*}
whenever $r\ge1/2$ and $0\le t\le1$.

\smallskip
\noindent\emph{Step 3: The near and far regions.}
Let $\tau>0$. By \eqref{eq:H-little-o}, choose $\delta\in(0,1)$ so
that $|H(\eta)|\le\tau t$ whenever $0\le t\le\delta$.
In geodesic polar coordinates about $\xi_0$, let $\varphi\in[0,2\pi)$ be the angular variable, so that $\dd\sigma=\sin t\,\dd t\,\dd\varphi$. Using $|H(\eta)|\le\tau t$ and $\sin t\le t$, and absorbing the angular factor $2\pi$ into $C$, the contribution from the cap $0\le t\le\delta$ is bounded by
\begin{align*}
  &\int_{\{t\le\delta\}}
  |D_vP_r(\xi_0,\eta)|\,|H(\eta)|\dd\sigma(\eta)\\
  &\qquad\le C\tau\varepsilon\int_0^\delta
  \frac{t^3}{(\varepsilon^2+t^2)^{5/2}}\dd t\\
  &\qquad= C\tau\int_0^{\delta/\varepsilon}
  \frac{s^3}{(1+s^2)^{5/2}}\dd s\\
  &\qquad\le C\tau\int_0^\infty
  \frac{s^3}{(1+s^2)^{5/2}}\dd s
  =\frac{2C}{3}\tau,
\end{align*}
where the substitution $t:=\varepsilon s$ was used.

For $t\ge\delta$ and $r\ge1/2$,
\begin{equation*}
  |r\xi_0-\eta|^2
  \ge2r(1-\cos\delta)
  \ge1-\cos\delta.
\end{equation*}
Put $\norm{H}_\infty:=\max_{\eta\in\Sph^2}|H(\eta)|$.
Equation~\eqref{eq:Poisson-tangent-derivative},
$|\ip{\eta}{v}|\le1$, and $1-r^2\le2(1-r)$ give
\begin{align*}
  &\int_{\{t\ge\delta\}}
  |D_vP_r(\xi_0,\eta)|\,|H(\eta)|\dd\sigma(\eta)\\
  &\qquad\le
  \frac{3(1-r)}
  {2\pi(1-\cos\delta)^{5/2}}
  \norm{H}_\infty\sigma(\Sph^2).
\end{align*}
Hence, for the fixed $\tau$, there exists $r_\tau\in(1/2,1)$ such
that the far contribution is at most $\tau$ whenever $r_\tau<r<1$.
Combining the two estimates gives
\begin{equation*}
  |D_vV_H(r,\xi_0)|
  \le\left(1+\frac{2C}{3}\right)\tau,
  \qquad r_\tau<r<1.
\end{equation*}
To verify the limit directly, let $\rho>0$ and take
$\tau:=\rho/(1+2C/3)$. Then
\begin{equation*}
  r_\tau<r<1
  \quad\Longrightarrow\quad
  |D_vV_H(r,\xi_0)|\le\rho,
\end{equation*}
so $D_vV_H(r,\xi_0)\to0$ as $r\to1^-$. For a general nonzero tangent
vector, linearity gives
\begin{equation*}
  D_vV_H(r,\xi_0)
  =\norm v\,D_{v/\norm v}V_H(r,\xi_0)\longrightarrow0.
\end{equation*}
The case $v=0$ holds by definition. Adding back the affine part proves
\eqref{eq:Poisson-boundary-convergence}.
\end{proof}

\begin{proposition}[The two-dimensional example]\label{prop:d2-seed}
The function $\phi$ in \eqref{eq:phi-seed} belongs to $\Psi_2$, but $\phi'(\pi/2)$ does not exist.
\end{proposition}

\begin{proof}
Membership in $\Psi_2$ was established above. Suppose, to the
contrary, that $\phi'(\pi/2)$ exists.

\smallskip
\noindent\emph{Step 1: Geometry at the equator.}
Choose $\xi_0\in\Sph^2$ with $\ip{e}{\xi_0}=0$ and set
$v:=-e\in T_{\xi_0}\Sph^2$. The map
$g(\xi):=\arccos\ip{e}{\xi}$ is smooth near $\xi_0$, and
\begin{equation*}
  D_vg(\xi_0)
  =-\frac{\ip{e}{v}}
  {\sqrt{1-\ip{e}{\xi_0}^2}}
  =1.
\end{equation*}
Write $\eta=\cos t\,\xi_0+\sin t\,u$, where
$u\in T_{\xi_0}\Sph^2$ and $\norm u=1$. Since
$\ip{e}{\xi_0}=0$, uniformly in $u$,
\begin{equation*}
  s(t,u):=g(\eta)-\frac\pi2
  =\arccos\bigl(\sin t\,\ip{e}{u}\bigr)-\frac\pi2
  =-t\ip{e}{u}+O(t^3).
\end{equation*}
Indeed, $\arccos z=\pi/2-z+O(z^3)$ as $z\to0$ and
$\sin t=t+O(t^3)$. Since $|\ip{e}{u}|\le1$, both remainders are
uniform over the unit vectors $u\in T_{\xi_0}\Sph^2$. In particular,
$|s(t,u)|\le Ct$ for all sufficiently small $t$, with $C$ independent
of $u$.

\smallskip
\noindent\emph{Step 2: Differentiability of the zonal function.}
The assumed existence of $\phi'(\pi/2)$ gives
\begin{equation*}
  \phi\left(\frac\pi2+s\right)
  =\phi\left(\frac\pi2\right)
  +\phi'\left(\frac\pi2\right)s+o(|s|).
\end{equation*}
Write the remainder as $R(s)$. Thus $R(s)/|s|\to0$ as $s\to0$.
After defining this quotient to be zero at $s=0$, the preceding
uniform bound on $s(t,u)$ gives
\begin{equation*}
  \sup_{\substack{u\in T_{\xi_0}\Sph^2\\ \norm u=1}}
  \frac{|R(s(t,u))|}{t}
  \le C\sup_{|y|\le Ct}\frac{|R(y)|}{|y|}
  \longrightarrow0.
\end{equation*}
Using $F(\eta)=\phi(g(\eta))$, the composition can now be expanded
uniformly in $u$:
\begin{align*}
  F(\eta)-F(\xi_0)
  &=\phi'\left(\frac\pi2\right)s(t,u)+R(s(t,u))\\
  &=-t\phi'\left(\frac\pi2\right)\ip{e}{u}
    +O(t^3)+o(t)\\
  &=t\ip{-\phi'(\pi/2)e}{u}+o(t).
\end{align*}
Thus $F$ is differentiable at $\xi_0$ in the sense of
Lemma~\ref{lem:Poisson-boundary}, with
\begin{equation*}
  b:=-\phi'\left(\frac\pi2\right)e,
  \qquad
  D_vF(\xi_0)=\ip{b}{v}=\phi'\left(\frac\pi2\right).
\end{equation*}

\smallskip
\noindent\emph{Step 3: Contradiction with the Poisson means.}
By \eqref{eq:U-zonal} and the chain rule,
\begin{align*}
  D_vU(r,\xi_0)
  &=\phi_r'(g(\xi_0))D_vg(\xi_0)
    =\phi_r'\left(\frac\pi2\right).
\end{align*}
All hypotheses of Lemma~\ref{lem:Poisson-boundary} are now satisfied.
It would make $\phi_r'(\pi/2)=D_vU(r,\xi_0)$ converge to the finite
number $D_vF(\xi_0)=\phi'(\pi/2)$, contradicting
\eqref{eq:phi-r-minus-infty}.
\end{proof}

\subsection{A strictly positive definite modification}

The Schoenberg sequence of $\phi$ is supported on one residue class, so $\phi$ is not strictly positive definite. For $0<s<1$, define
\begin{equation*}
  q_s(\theta)
  :=(1-s)\sum_{n=0}^\infty s^nP_n(\cos\theta)
  =\frac{1-s}{\sqrt{1-2s\cos\theta+s^2}}.
\end{equation*}
The second equality is the Legendre generating function~\cite[Eq.~18.12.11]{DLMF}. Thus $q_s$ is real analytic on a neighborhood of $[0,\pi]$ and belongs to $\Psi_2^+$, with every Schoenberg coefficient positive. For $0<\epsilon<1$, define
\begin{equation*}
  \phi_*(\theta):=(1-\epsilon)\phi(\theta)+\epsilon q_s(\theta).
\end{equation*}
Then $\phi_*\in\Psi_2^+$, every Schoenberg coefficient of $\phi_*$ is positive, and $\phi_*'(\pi/2)$ does not exist; otherwise subtraction of the analytic term would make $\phi'(\pi/2)$ exist.

\section{From dimension two to all even dimensions}\label{sec:even}

To obtain examples in all even dimensions, we first apply turning bands to the function $\phi_*$  in $\Psi_2$ and obtain an example in every even dimension. Iteration of the spherical mont\'ee then increases the differentiability at zero. Lemmas~\ref{lem:turning-singularity} and~\ref{lem:montee-singularity} show that the missing derivative remains absent at $\pi/2$ after each operation.

\subsection{The case \texorpdfstring{$k=0$}{k=0}}

Let $\beta:=(\beta_n)_{n\ge0}$ be the $2$-Schoenberg sequence of $\phi_*$; by construction,
\begin{equation*}
  \beta_n>0\qquad(n\ge0).
\end{equation*}
For $M\ge1$, define
\begin{equation*}
  \beta^{(M)}:=T^{M-1}\beta,
  \qquad
  F_M(\theta):=\Sd_{2M}(\beta^{(M)};\theta).
\end{equation*}
In particular, $\beta^{(1)}=\beta$ and $F_1=\phi_*$.
Let
\begin{equation*}
  B_M:=F_M(0)=\sum_{n=0}^\infty\beta_{n+M-1}>0,
  \qquad
  \chi_{2M}:=B_M^{-1}F_M.
\end{equation*}
Because all coefficients are positive, $\chi_{2M}\in\Psi_{2M}^+$.

The turning-bands identity becomes
\begin{equation*}
  F_M(\theta)
  =\beta_{M-1}+\cos\theta\,F_{M+1}(\theta)
  +\frac{\sin\theta}{2M}F_{M+1}'(\theta).
\end{equation*}

\begin{proposition}\label{prop:even-lift}
For every $M\ge1$,
\begin{equation*}
  F_M^{(M)}\left(\frac\pi2\right)
  \quad\text{does not exist}.
\end{equation*}
Equivalently, the same is true for $\chi_{2M}$.
\end{proposition}

\begin{proof}
For $M=1$, $F_1=\phi_*$ and the result follows from Proposition~\ref{prop:d2-seed} and the analytic perturbation. Assume $M\ge2$.
For $1\le j\le M$, one has $F_j(0)>0$ and
$F_j/F_j(0)\in\Psi_{2j}$. Applying
Theorem~\ref{thm:main-smoothness} to the normalized function and then
multiplying by $F_j(0)$ gives
\begin{equation*}
  F_j\in C^{j-1}(0,\pi),
  \qquad j=1,\ldots,M.
\end{equation*}
Suppose that $F_M^{(M)}(\pi/2)$ existed. For each integer $p$ with $2\le p\le M$, apply Lemma~\ref{lem:turning-singularity} with
\begin{equation*}
  F:=F_{p-1},\qquad G:=F_p,\qquad r:=p.
\end{equation*}
Indeed, $\beta^{(p)}=T\beta^{(p-1)}$, and
\begin{equation*}
  F_{p-1}=\Sd_{2(p-1)}(\beta^{(p-1)};\cdot),
  \qquad
  F_p=\Sd_{2p}(T\beta^{(p-1)};\cdot).
\end{equation*}
Thus $(F_{p-1},F_p)$ is precisely the pair $(F,G)$ in
Proposition~\ref{prop:turning-bands} with $d=2(p-1)$.
The hypothesis $G\in C^{r-1}$ is exactly
$F_p\in C^{p-1}(0,\pi)$. The contrapositive of the lemma therefore
gives, for each $p$, the implication
$F_p^{(p)}(\pi/2)$ exists $\Longrightarrow$
$F_{p-1}^{(p-1)}(\pi/2)$ exists.
Starting from the contradictory assumption and taking
$p=M,M-1,\ldots,2$ gives the explicit chain
\begin{align*}
  F_M^{(M)}\left(\frac\pi2\right)\text{ exists}
  &\Longrightarrow
  F_{M-1}^{(M-1)}\left(\frac\pi2\right)\text{ exists}\\
  &\Longrightarrow\cdots\Longrightarrow
  F_1'\left(\frac\pi2\right)\text{ exists}.
\end{align*}
The last statement contradicts the base case. Finally,
$\chi_{2M}^{(M)}(\pi/2)=B_M^{-1}F_M^{(M)}(\pi/2)$ whenever either
derivative exists, so multiplication by $B_M^{-1}\ne0$ does not alter
the conclusion.
\end{proof}

For $d=2M$, Theorem~\ref{thm:main-smoothness} guarantees $C^{M-1}(0,\pi)$, whereas Proposition~\ref{prop:even-lift} shows that the derivative of order $M$ need not exist. Thus the bound for $k=0$ is optimal in every even dimension.

\subsection{Arbitrary even differentiability at the origin}

\begin{proof}[Proof of Theorem~\ref{thm:main-even}]
Fix $d=2m\ge2$ and $k\ge0$. Define
\begin{equation*}
  M:=m+k,
  \qquad
  D:=2M=d+2k.
\end{equation*}
Each application of $\IS$ lowers the dimension by two, raises the pole
regularity by two, and raises the order of the missing equatorial derivative
by one.
By its definition, $\chi_{2M}\in\Psi_D^+$ and every Schoenberg coefficient
of $\chi_{2M}$ is positive. Write $(\gamma_n)_{n\ge0}$ for its
$D$-Schoenberg sequence. Proposition~\ref{prop:even-lift} gives
\begin{equation}\label{eq:chi-missing}
  \chi_{2M}^{(M)}\left(\frac\pi2\right)
  \quad\text{does not exist}.
\end{equation}
\smallskip
\noindent\emph{Step 1: Construction and iteration.}
To iterate the mont\'ee using Proposition~\ref{prop:montee-positive}, the starting function must be nonnegative. Since $\chi_{2M}$ is continuous, choose
$C>\max\{0,-\min_{0\le\theta\le\pi}\chi_{2M}(\theta)\}$ and set
$H_0(\theta):=(\chi_{2M}(\theta)+C)/(1+C)$. If $(h_{0,n})_{n\ge0}$ is its
$D$-Schoenberg sequence, then
\begin{align*}
  &H_0(\theta)>0\ (0\le\theta\le\pi),\qquad H_0(0)=1,\\
  &h_{0,0}=\frac{\gamma_0+C}{1+C}>0,\qquad
   h_{0,n}=\frac{\gamma_n}{1+C}>0\quad(n\ge1),\\
  &H_0^{(M)}\left(\frac\pi2\right)\text{ exists}
   \Longrightarrow
   \chi_{2M}^{(M)}\left(\frac\pi2\right)
   =(1+C)H_0^{(M)}\left(\frac\pi2\right)\text{ exists}.
\end{align*}
The first two lines and the strictness criterion from Section~2.1 show that
$H_0\in\Psi_D^+$.
Thus \eqref{eq:chi-missing} remains true with $H_0$ in place of $\chi_{2M}$.

Define $H_{j+1}:=\IS H_j$ for $0\le j<k$, and put
$D_j:=D-2j=2(M-j)$. For $0\le j<k$,
$D_j=d+2(k-j)\ge4$.
We prove by induction that $H_j\in\Psi_{D_j}^+$, that $H_j\ge0$,
and that its Schoenberg coefficients are positive in all sufficiently
large degrees. These assertions hold for $j=0$. Suppose they hold for
some $j<k$. Proposition~\ref{prop:montee-positive} gives $A(H_j)>0$,
so $H_{j+1}$ is well defined. If $(h_{j,n})_{n\ge0}$ denotes the
$D_j$-Schoenberg sequence of $H_j$, Proposition~\ref{prop:montee-positive}
and \eqref{eq:montee-coeff} give
\begin{align*}
  &H_{j+1}\in\Psi_{D_j-2}^+=\Psi_{D_{j+1}}^+,\\
  &H_{j+1}(\theta)
   =\frac{\displaystyle\int_\theta^\pi\sin t\,H_j(t)\dd t}
   {A(H_j)}\ge0,\qquad H_{j+1}(0)=1,\\
  &h_{j+1,n}
   =\frac{(D_j-2)h_{j,n-1}}
   {n(n+D_j-3)A(H_j)},\qquad n\ge1.
\end{align*}
Thus $h_{j,n-1}>0$ for all sufficiently large $n$ implies
$h_{j+1,n}>0$ for all sufficiently large $n$. The induction therefore
yields
\begin{equation*}
  H_k\in\Psi_{D-2k}^+=\Psi_{2m}^+=\Psi_d^+.
\end{equation*}

\smallskip
\noindent\emph{Step 2: Regularity at the pole.}
For $0\le j\le k$, define the even continuation
\begin{equation*}
  H_{j,e}(t):=H_j(|t|),\qquad -\pi<t<\pi.
\end{equation*}
The function $H_{0,e}$ is continuous at $0$. For $0\le j<k$,
Lemma~\ref{lem:montee-pole}, in dimension $D_j$ with $q=j$, gives
\begin{align*}
  &H_{j,e}\text{ is $2j$ times differentiable at $0$}\\
  &\qquad\Longrightarrow
  H_{j+1,e}\text{ is $2j+2$ times differentiable at $0$}.
\end{align*}
Starting with $j=0$ and iterating this implication yields
$2j$-fold differentiability of $H_{j,e}$ at $0$ for every
$0\le j\le k$.
The assertion with $j=k$ is part~\textup{(i)}.

\smallskip
\noindent\emph{Step 3: Interior differentiability and sharpness.}
The calculation in Step~1 and \eqref{eq:chi-missing} show that
$H_0^{(M)}(\pi/2)$ does not exist. If
$H_j^{(M+j)}(\pi/2)$ does not exist, the construction gives
$A(H_j)>0$, and we apply
Lemma~\ref{lem:montee-singularity} with
$f:=H_j$, $r:=M+j$, and $\theta_0:=\pi/2$. Since
$H_{j+1}=\IS H_j$, the resulting chain is
\begin{align*}
  H_0^{(M)}\left(\frac\pi2\right)\text{ does not exist}
  &\Longrightarrow
  H_1^{(M+1)}\left(\frac\pi2\right)\text{ does not exist}\\
  &\Longrightarrow\cdots\Longrightarrow
  H_k^{(M+k)}\left(\frac\pi2\right)\text{ does not exist}.
\end{align*}
Since $M+k=m+2k$, the last statement proves part~\textup{(iii)}.

Theorem~\ref{thm:main-smoothness}, applied to $H_k\in\Psi_d$ and the pole regularity proved above, gives
\begin{equation*}
  H_k\in C^{2k+\lfloor(2m-1)/2\rfloor}(0,\pi)
  =C^{2k+m-1}(0,\pi),
\end{equation*}
which is part~\textup{(ii)}.

Suppose that the even continuation of $H_k$ were $2k+2$ times differentiable at zero. Theorem~\ref{thm:main-smoothness}, with $k+1$ in place of $k$, would give
\begin{align*}
  H_{k,e}\text{ is $2k+2$ times differentiable at $0$}
  &\Longrightarrow
  H_k\in C^{2k+2+\lfloor(2m-1)/2\rfloor}(0,\pi)\\
  &=C^{2k+m+1}(0,\pi)\\
  &\Longrightarrow
  H_k^{(2k+m)}\left(\frac\pi2\right)\text{ exists}.
\end{align*}
This contradicts part~\textup{(iii)} and proves part~\textup{(iv)}.

Finally, suppose that $H_k\in\Psi_{d+1}$. Part~\textup{(i)} supplies
the required $2k$-fold differentiability at zero. Hence
Theorem~\ref{thm:main-smoothness} in dimension $d+1=2m+1$ would give
\begin{align*}
  H_k\in\Psi_{d+1}
  &\Longrightarrow
  H_k\in C^{2k+\lfloor d/2\rfloor}(0,\pi)
   =C^{2k+m}(0,\pi)\\
  &\Longrightarrow
  H_k^{(2k+m)}\left(\frac\pi2\right)\text{ exists},
\end{align*}
again contradicting part~\textup{(iii)}. This proves part~\textup{(v)}.
Define $\psi_{d,k}:=H_k$ to complete the proof.
\end{proof}

\section{The optimal differentiability index}

For $d\ge1$ and $k\ge0$, let $r_*(d,k)$ be the largest $r\in\Nzero$ for which every $\psi\in\Psi_d$ whose even continuation has $2k$ derivatives at $0$ belongs to $C^r(0,\pi)$. The even-dimensional construction and the known odd-dimensional examples give the following formula.

\begin{corollary}
For every $d\ge1$ and $k\ge0$,
\begin{equation}\label{eq:optimal-index}
  r_*(d,k)=2k+\left\lfloor\frac{d-1}{2}\right\rfloor.
\end{equation}
More explicitly,
\begin{center}
\begin{tabular}{@{}ccc@{}}
\toprule
Dimension & Guaranteed order & First non-guaranteed derivative \\
\midrule
$d=2m+1$, $m\ge0$ & $2k+m$ & $2k+m+1$ \\
$d=2m$, $m\ge1$   & $2k+m-1$ & $2k+m$ \\
\bottomrule
\end{tabular}
\end{center}
In the even-dimensional construction the first non-guaranteed derivative
fails at the fixed point $\pi/2$. The first row includes the circle $d=1$
by taking $m=0$; the odd-dimensional examples are given in
Ziegel~\cite[Sec.~4]{Ziegel2014} and
Tr\"ubner--Ziegel~\cite[Sec.~4]{TrubnerZiegel2017}.
\end{corollary}

\begin{proof}
The lower bound in \eqref{eq:optimal-index} is Theorem~\ref{thm:main-smoothness}. The upper bound in odd dimensions is due to Ziegel~\cite[Sec.~4]{Ziegel2014} for $k=0$ and Tr\"ubner--Ziegel~\cite[Sec.~4]{TrubnerZiegel2017} for general $k$. The upper bound in even dimensions is Theorem~\ref{thm:main-even}.
\end{proof}


\begin{thebibliography}{99}

\bibitem{BeatsonZuCastell2017}
R.~K. Beatson and W.~zu Castell,
\emph{Dimension hopping and families of strictly positive definite zonal basis functions on spheres},
J. Approx. Theory \textbf{221} (2017), 22--37.

\bibitem{ChenMenegattoSun2003}
D.~Chen, V.~A. Menegatto, and X.~Sun,
\emph{A necessary and sufficient condition for strictly positive definite functions on spheres},
Proc. Amer. Math. Soc. \textbf{131} (2003), 2733--2740.

\bibitem{DaiXu2013}
F.~Dai and Y.~Xu,
\emph{Approximation Theory and Harmonic Analysis on Spheres and Balls},
Springer Monographs in Mathematics, Springer, 2013.

\bibitem{DLMF}
NIST Digital Library of Mathematical Functions,
Release 1.2.7, \url{https://dlmf.nist.gov/}.

\bibitem{EmeryPeronPorcu2022}
X.~Emery, A.~P. Peron, and E.~Porcu,
\emph{Dimension walks on hyperspheres},
Comput. Appl. Math. \textbf{41} (2022), Paper No.~199.

\bibitem{Gneiting1999}
T.~Gneiting,
\emph{On the derivatives of radial positive definite functions},
J. Math. Anal. Appl. \textbf{236} (1999), 86--93.

\bibitem{Gneiting2013}
T.~Gneiting,
\emph{Strictly and non-strictly positive definite functions on spheres},
Bernoulli \textbf{19} (2013), 1327--1349.

\bibitem{Jager2019}
J.~J\"ager,
\emph{A note on the derivatives of isotropic positive definite functions on the Hilbert sphere},
SIGMA \textbf{15} (2019), 081, 7 pages.

\bibitem{Matheron1965}
G.~Matheron,
\emph{Les Variables R\'egionalis\'ees et Leur Estimation},
Masson et Cie, Paris, 1965.

\bibitem{Matheron1972}
G.~Matheron,
\emph{Quelques aspects de la mont\'ee},
Note G\'eostatistique 120, Centre de G\'eostatistique, Fontainebleau, France, 1972.

\bibitem{Schoenberg1942}
I.~J. Schoenberg,
\emph{Positive definite functions on spheres},
Duke Math. J. \textbf{9} (1942), 96--108.

\bibitem{TrubnerZiegel2017}
M.~Tr\"ubner and J.~F. Ziegel,
\emph{Derivatives of isotropic positive definite functions on spheres},
Proc. Amer. Math. Soc. \textbf{145} (2017), 3017--3031.

\bibitem{Ziegel2014}
J.~F. Ziegel,
\emph{Convolution roots and differentiability of isotropic positive definite functions on spheres},
Proc. Amer. Math. Soc. \textbf{142} (2014), 2063--2077.

\end{thebibliography}
\end{document}